\documentclass{amsart}
\usepackage{bbold}
\usepackage{amsmath}
\usepackage{amssymb}
\usepackage{amsfonts}

\newcommand{\Abs}[1]{\left\vert #1 \right\vert}
\newcommand{\abs}[1]{\vert #1 \vert}
\newcommand{\bigabs}[1]{\bigl\vert #1 \bigr\vert}

\newcommand{\norm}[1]{\left\Vert #1 \right\Vert}

\newcommand{\bignorm}[1]{\bigl\Vert #1 \bigr\Vert}

\newcommand{\Nmin}{N_{\mathrm{min}}}

\newcommand{\Lmin}{L_{\mathrm{min}}}
\newcommand{\Lmed}{L_{\mathrm{med}}}
\newcommand{\Lmax}{L_{\mathrm{max}}}

\newcommand{\C}{\mathbb{C}}
\newcommand{\N}{\mathbb{N}}

\newcommand{\R}{\mathbb{R}}

\newcommand{\Innerprod}[2]{\left\langle \, #1 , #2 \, \right\rangle}
\newcommand{\innerprod}[2]{\langle \, #1 , #2 \, \rangle}

\newcommand{\angles}[1]{\langle #1 \rangle}

\DeclareMathOperator{\im}{Im}
\DeclareMathOperator{\re}{Re}

\DeclareMathOperator{\sgn}{sgn}

\newtheorem{theorem}{Theorem}

\newtheorem{lemma}{Lemma}
\newtheorem{corollary}{Corollary}

\theoremstyle{definition}
\newtheorem{definition}{Definition}

\theoremstyle{remark}
\newtheorem{remark}{Remark}

\title[Spatial analyticity for 2d DKG]{On the radius of spatial analyticity for solutions of the Dirac-Klein-Gordon equations in two space dimensions}

\author[S.~Selberg]{Sigmund Selberg}

\address{Department of Mathematics, University of Bergen, PO Box 7803, 5020 Bergen, Norway}
\email{sigmund.selberg@uib.no}
\date{}

\begin{document}

\begin{abstract}
We consider the initial value problem for the Dirac-Klein-Gordon equations in two space dimensions. Global regularity for  $C^\infty$ data was proved by Gr\"unrock and Pecher. Here we consider analytic data, proving that if the initial radius of analyticity is $\sigma_0 > 0$, then for later times $t > 0$  the radius of analyticity obeys a lower bound $\sigma(t) \ge \sigma_0 \exp(-At)$. This provides information about the possible dynamics of the complex singularities of the holomorphic extension of the solution at time $t$. The proof relies on an analytic version of Bourgain's Fourier restriction norm method, multilinear space-time estimates of null form type and an approximate conservation of charge.
\end{abstract}

\maketitle

\section{Introduction}

We consider the Cauchy problem for the Dirac-Klein-Gordon (DKG) equations in two space dimensions,
\begin{equation}\label{DKG2d}
\left\{
\begin{alignedat}{2}
  (-i\partial_t - i\alpha \cdot \nabla + M\beta)\psi &= \phi\beta\psi,& \qquad \psi(0,x) &= \psi_0(x),
  \\
  (\partial_t^2 - \Delta + m^2) \phi &= \innerprod{\beta \psi}{\psi},& \qquad (\phi,\partial_t\phi)(0,x) &= (\phi_0,\phi_1)(x),
\end{alignedat}
\right.
\end{equation}
where the unknowns $\psi$ (the Dirac spinor) and $\phi$ (the meson field) are functions of $(t,x) \in \R \times \R^2$ and take values in $\C^2$ and $\R$, respectively, and $\psi = (\psi_1,\psi_2)^t$ is considered as a column vector upon which the Dirac matrices (in fact, the Pauli matrices)
\[
  \alpha^{1} =
  \begin{pmatrix}
    0 & 1  \\
    1 & 0
  \end{pmatrix},
  \quad
  \alpha^2 =
  \begin{pmatrix}
    0 & -i  \\
    i & 0
  \end{pmatrix}
  \quad
  \text{and}
  \quad
  \beta =
  \begin{pmatrix}
    1 & 0  \\
    0 & -1
  \end{pmatrix}
\]
may act. The standard inner product on $\C^2$ is denoted $\innerprod{\cdot}{\cdot}$. We write $x=(x^1,x^2)$, $\partial_j = \frac{\partial}{\partial x^j}$, $\nabla = (\partial_1,\partial_2)$, $\Delta = \partial_1^2 + \partial_2^2$ and $\alpha \cdot \nabla = \alpha^1 \partial_1 + \alpha^2 \partial_2$. The masses $M$ and $m$ are given real constants. We shall assume $m > 0$.

In particle physics, DKG arises as a model for forces between nucleons, mediated by mesons; see \cite{Bjorken1964}. The well-posedness of the Cauchy problem in space dimensions $d \le 3$ with data in the family of Sobolev spaces $H^s(\R^d) = W^{s,2}(\R^d)$ has been extensively studied; see~\cite{Chadam:1973, DFS2007, DFS2007b, GP2010, Machihara:2010, Wang:2015, Herr:2014, Herr:2016} and the references therein.

Our aim in this article is to add to the large-data global regularity theory in space dimension $d=2$. Global regularity for $C^\infty(\R^2)$ data was proved by Gr\"unrock and Pecher \cite{GP2010}. Our focus here is on spatial analyticity, with a uniform radius of analyticity $\sigma(t) > 0$ for each time $t$. By this we mean that the solution at time $t$ has a holomorphic extension to the complex strip
\[
  S_{\sigma} = \left\{ x + iy \in \C^2 \colon \text{$x,y \in \R^2$ and $\abs{y_1},\abs{y_2} < \sigma$} \right\}
\]
with $\sigma=\sigma(t)$.

Heuristically, the picture one should have in mind is that $\sigma(t)$ is the distance from $\R^2_x$ to the nearest complex singularity of the holomorphic extension of the solution at time $t$. We will prove a lower bound
\[
\sigma(t) \ge \sigma_0 \exp(-At)
\]
as $t \to \infty$, providing some information about the possible dynamics of the complex singularities.

The proof of global $C^\infty$ regularity in \cite{GP2010} makes use of Bourgain's Fourier restriction norm method, and a key motivation behind the present paper was to investigate to which extent the analytic version of this method---introduced by Bourgain in \cite[Section 8]{Bourgain1993b}---can yield refined information about the regularity of the solution for analytic data. A further motivation was a recent result of Cappiello, D'Ancona and Nicola \cite{CDN2014} (see also \cite{Alin:1984}) on persistence of spatial analyticity for $C^\infty$ solutions of semilinear symmetric hyperbolic systems, which in the special case of DKG\footnote{DKG can be written as a semilinear symmetric hyperbolic system with unknown $u = (\psi_1, \psi_2, \phi,\partial_1 \phi, \partial_2 \phi, \partial_t \phi)^\intercal$.} yields a lower bound
\[
  \sigma(t) \ge \sigma_0 \exp\left(-A\int_0^t \left(1+\norm{\psi(s)}_{L^\infty}+\norm{\phi(s)}_{L^\infty}+\norm{\partial \phi(s)}_{L^\infty}\right) \, ds\right).
\]
This is weaker than our lower bound $\sigma(t) \ge \sigma_0 \exp(-At)$, since the best estimate known on the $L^\infty$ norm of the solution of \eqref{DKG2d} appears to be $O(\exp(Ct))$, which can be obtained from the global existence proof in \cite{GP2010}, hence one would get a double exponential decay rate $\sigma(t) \ge \sigma_0 \exp(-A\exp(Ct))$.

The investigation of spatially uniform lower bounds on the radius of analyticity for nonlinear evolutionary PDE was initiated by Kato and Masuda \cite{KM1986}, and by now there is an extensive catalog of results along these lines for various nonlinear PDE, including the Kadomtsev-Petviashvili equation \cite{Bourgain1993b}, the (generalized) Korteweg-de Vries equation \cite{H1991,BGK2005,Selberg2017}, the Euler equations \cite{KV2009}, the cubic Szeg\H o equation \cite{Gerard2015} and the nonlinear Schr\"odinger equation \cite{H1990,BGK2006,Achenef2017}.

Since the radius of analyticity can be related to the asymptotic decay of the Fourier transform, it is natural to use Fourier methods to study the type of problem outlined above. We take data in the analytic Gevrey class $G^{\sigma,s} = G^{\sigma,s}(\R^2)$, defined for $\sigma > 0$ and $s \in \R$ by
\[
  G^{\sigma,s}(\R^2) = \left\{ f \in L^2(\R^2) \colon \norm{f}_{G^{\sigma,s}} < \infty \right\},
  \qquad
  \norm{f}_{G^{\sigma,s}} = \bignorm{ e^{\sigma\norm{\xi}} \angles{\xi}^s \widehat f(\xi) }_{L^2_\xi}.
\]
Here we denote, for $\xi = (\xi_1,\xi_2) \in \R^2$,
\[
  \norm{\xi} = \abs{\xi_1} + \abs{\xi_2},
  \qquad
  \abs{\xi} = \left( \xi_1^2 + \xi_2^2 \right)^{1/2},
  \qquad
  \angles{\xi} = (1+\abs{\xi}^2)^{1/2},
\]
and
\[
  \widehat f(\xi) = \mathcal F f(\xi) = \int_{\R^2} e^{-ix\cdot\xi} f(x) \, dx
\]
is the Fourier transform. Note that $G^{\sigma,s} = \mathcal F^{-1} ( e^{-\sigma\norm{\cdot}} \angles{\cdot}^{-s} L^2 )$ is isometrically isomorphic to $L^2$ and hence is a Banach space. We recall the fact that every $f \in G^{\sigma,s}$ has a uniform radius of analyticity $\sigma$, that is, $f$ has a holomorphic extension to  $S_\sigma$ (for a proof see, e.g., \cite{Selberg2018}).

Our main result is the following.

\begin{theorem}\label{Thm1}
Consider \eqref{DKG2d} with $m > 0$. Let $\sigma_0 > 0$. Given initial data
\begin{equation}\label{data}
  (\psi_0,\phi_0,\phi_1) \in G^{\sigma_0,0}(\R^2;\C^2) \times G^{\sigma_0,1/2}(\R^2;\R) \times G^{\sigma_0,-1/2}(\R^2;\R),
\end{equation}
let $(\psi,\phi)$ be the unique global $C^\infty$ solution of \eqref{DKG2d}, as obtained in \cite{GP2010}. Then for any $T > 0$ we have
\[
  (\psi,\phi,\partial_t \phi) \in C([-T,T];G^{\sigma(T),0} \times G^{\sigma(T),1/2} \times G^{\sigma(T),-1/2}),
\]
where
\[
  \sigma(T) = \sigma_0 e^{-AT}
\]
for some constant $A > 0$ depending on $\sigma_0$ and the norm of the data. Thus, for any time $t \in \R$, the solution has a uniform radius of analyticity at least $\sigma(\abs{t})$.
\end{theorem}

We have no reason to expect that this bound is optimal, but it does appear to be the best possible with the technique used in the proof, which is based on an analytic version of Bourgain's Fourier restriction norm method, multilinear space-time estimates of null form type and an approximate version of the conservation of charge
\begin{equation}\label{ChargeCons}
  \int_{\R^2} \Abs{\psi(t,x)}^2 \, dx = \mathrm{const}.
\end{equation}

We now describe in more detail the method of proof. The point of departure is the observation that the norm on $G^{\sigma,s}$ is obtained from the Sobolev norm
\[
  \norm{f}_{H^{s}} = \bignorm{\angles{\xi}^s \widehat f(\xi)}_{L^2_\xi}
\]
by the substitution
\[
  f \longrightarrow e^{\sigma\norm{D}}f = \mathcal F^{-1} \left( e^{\sigma\norm{\cdot}} \mathcal Ff \right).
\]
Applying the same substitution in the setting of Bourgain's Fourier restriction norm method, the space $X^{s,b}$ then yields the analytic space $X^{\sigma,s,b}$. This idea was used by Bourgain \cite[Theorem 8.12]{Bourgain1993b} to study spatial analyticity for the Kadomtsev-Petviashvili equation, but the argument applies to a class of dispersive PDE in general, as discussed in \cite{Selberg2018}. In brief summary, the consequences that can be abstracted from Bourgain's argument are the following.
\begin{itemize}
\item[(B1)] If local well-posedness of some nonlinear dispersive PDE can be proved for $H^s$ initial data by a contraction argument in $X^{s,b}$, then one also has local well-posedness for data in $G^{\sigma_0,s}$ for any $\sigma_0 > 0$.
\item[(B2)] If, moreover, the solution extends globally in time (so the $H^s$ norm does not blow up in finite time), then the solution remains spatially analytic for all time, but no lower bound is obtained on $\sigma(t) > 0$ as $t \to \infty$.
\end{itemize}
An additional observation, proved in \cite{Selberg2018}, is that:
\begin{itemize}
\item[(B3)] If the $H^s$ norm is conserved, then $\sigma(t) \ge \sigma_0 \exp(-At)$ is obtained.
\end{itemize}

We emphasize that (B3) does not apply to DKG, since there is no conservation law for the field $\phi$. Thus a more involved argument is needed to prove our main result. The first and easiest step of the proof is to use the idea behind (B1) to obtain a local well-posedness result for data \eqref{data}, analogous to the local result from \cite{GP2010} with $H^s$ data. To reach any time $T > 0$, we then iterate the local result, and to control the growth of the data norms in each step we rely on an approximate conservation law for $\psi(t,\cdot)$ in $G^{\sigma,0}$, involving the parameter $\sigma > 0$ and reducing to the conservation law \eqref{ChargeCons} in the limit $\sigma \to 0$. Superficially, this parallels the approach used by Tesfahun and the author in \cite{ST2015} for the 1d DKG problem, for which an algebraic lower bound was obtained, but the function spaces and estimates are much more involved in the 2d case. See Remark \ref{AlgebraicRemark} below for an explanation of why we only get an exponential lower bound instead of an algebraic one in 2d. In 3d, on the other hand, global $C^\infty$ regularity for large data remains an open problem.

We now turn to the proof of Theorem \ref{Thm1}. Since $m > 0$, we may assume $m=1$ by a rescaling. 

The paper is organized as follows. In the next section we reformulate the system in a way which makes it easy to see the null structure. In Section \ref{MainProof} we state the analytic local existence theorem and the approximate conservation law and prove that they imply the main result, Theorem \ref{Thm1}. In Section \ref{Spaces} we introduce the function spaces that we use. In Section \ref{Spacetime} we prove some multilinear space-time estimates of null-form type, which are then used to prove the local existence in Section \ref{Local} and the approximate conservation law in Section \ref{Approx}.

\section{Reformulation of the system}

Set $D = \nabla/i$. For a given function $\xi \mapsto h(\xi)$ on $\R^2$ we denote by $h(D)$ the Fourier multiplier defined by
\[
  h(D) f\ = \mathcal F^{-1} \bigl( h(\xi) \mathcal F f(\xi) \bigr).
\]
Using the Dirac projections
\[
  \Pi_\pm = \Pi(\pm D), \qquad \Pi(\xi) = \frac12 \left( I + \frac{\xi}{\abs{\xi}}  \cdot \alpha \right)
\]
we now write
\[
  \psi = \psi_+ + \psi_-, \quad \text{where} \quad \psi_\pm = \Pi_\pm\psi.
\]
Further we set
\[
  \phi = \frac12 \left( \phi_+ + \phi_- \right), \quad \text{where} \quad \phi_\pm = \phi \pm i\angles{D}^{-1}\partial_t\phi,
\]
and note that $\phi = \re\phi_+$, since $\phi$ is real valued (hence so is $\angles{D}^{-1}\partial_t\phi$). Since
\[
  \abs{D}\Pi_+-\abs{D}\Pi_-=- i\alpha \cdot \nabla + \beta,
\]
one then obtains the following formulation of \eqref{DKG2d} (with $m=1)$:
\begin{equation}\label{DKG}
\left\{
\begin{alignedat}{2}
  (-i\partial_t + \abs{D})\psi_+ &= \Pi_+\bigl( - M\beta\psi + (\re\phi_+)\beta\psi\bigr),& \qquad \psi_+(0,x) &= f_+(x),
  \\
  (-i\partial_t - \abs{D})\psi_- &= \Pi_-\bigl(- M\beta\psi + (\re\phi_+)\beta\psi\bigr),& \qquad \psi_-(0,x) &= f_-(x),
  \\
  (-i\partial_t + \angles{D}) \phi_+ &= \angles{D}^{-1} \innerprod{\beta \psi}{\psi},& \qquad \phi_+(0,x) &= g_+(x),
\end{alignedat}
\right.
\end{equation}
where $f_\pm = \Pi_\pm\psi_0$ and $g_+ = \phi_0+i\angles{D}^{-1}\phi_1$.

As shown in \cite{DFS2007}, each bilinear term in \eqref{DKG} has a spinorial null structure encoded in the estimate
\begin{equation}\label{NullEstimate}
  \Pi(-s_2\eta)\Pi(s_1\xi) = O(\angle(s_1\xi,s_2\eta)),
\end{equation}
where $\xi,\eta \in \R^2$, $s_1,s_2 \in \{+,-\}$ and $-s_1$ denotes the reverse sign of $s_1$. This estimate will be used in tandem with the sign-reversing identity
\begin{equation}\label{Commutation}
  \Pi(\xi)\beta = \beta\Pi(-\xi).
\end{equation}

\section{Proof of the main theorem}\label{MainProof}

In this section we first state the analytic local existence theorem and the approximate conservation law, and then we show that they imply the main result, Theorem \ref{Thm1}.

We start with the local existence result (the proof is given in Section \ref{Local}).

\begin{theorem}\label{LWPa}
There exists a constant $c_0 > 0$ such that for any $\sigma_0 \ge 0$ and any data
\begin{equation}\label{Splitdata}
  (f_+,f_-,g_+) \in X_0 := G^{\sigma_0,0}(\R^2;\C) \times G^{\sigma_0,0}(\R^2;\C) \times G^{\sigma_0,1/2}(\R^2;\C),
\end{equation}
the Cauchy problem \eqref{DKG} has a unique solution
\[
  (\psi_+,\psi_-,\phi_+) \in C([-\delta,\delta];X_0)
\]
on $(-\delta,\delta) \times \R^2$, where
\[
  \delta = \frac{c_0}{1+\norm{f_+}_{G^{\sigma_0,0}}^2 + \norm{f_-}_{G^{\sigma_0,0}}^2 + \norm{g_+}_{G^{\sigma_0,1/2}}^2}.
\]
\end{theorem}

\begin{remark} The uniqueness is immediate since the solution is certainly $C^\infty$.
\end{remark}

\begin{remark}\label{AlgebraicRemark} If the dependence of the local existence time in Theorem \ref{LWPa} could be improved to
\[
  \delta = \frac{c_0}{1+\norm{f_+}_{G^{\sigma_0,0}}^2 + \norm{f_-}_{G^{\sigma_0,0}}^2 + \norm{g_+}_{G^{\sigma_0,1/2}}^\rho}
\]
for some $\rho < 2$, then the argument in subsection \ref{MainProof} below would give an algebraic lower bound on $\sigma(t)$ instead of an exponential one. But in order to get the improved existence time we would need to improve the estimate \eqref{Aest} used in the proof of the local existence theorem, more precisely the factor $\delta^{1/2}$ in that estimate would have to be replaced by $\delta^{1/\rho}$, and in view of \eqref{DifficultEstimate} this does not seem possible using the (sharp) estimates in Theorem \ref{TrilinearTheorem}.
\end{remark}

The conservation of charge
\[
  \int \abs{\psi(t,x)}^2 \, dx = \int \left( \abs{\psi_+(t,x)}^2 + \abs{\psi_-(t,x)}^2 \right) \, dx = \mathrm{const.}
\]
does not hold for $\Psi = e^{\sigma\norm{D}}\psi$ with $\sigma > 0$, but we can nevertheless obtain an approximate conservation law. Indeed, we have the following (proved in Section \ref{Approx}).

\begin{theorem}\label{ApproxCons} Let $\sigma_0 > 0$. Consider the local solution from Theorem \ref{LWPa}, with time of existence $\delta > 0$, and define
\begin{align*}
  \mathfrak M_\sigma(t) &= \norm{\psi_+(t,\cdot)}_{G^{\sigma,0}}^2 + \norm{\psi_-(t,\cdot)}_{G^{\sigma,0}}^2,
  \\
  \mathfrak N_\sigma(t) &= \norm{\phi_+(t,\cdot)}_{G^{\sigma,1/2}},
\end{align*}
for $t \in [0,\delta]$ and $\sigma \in [0,\sigma_0]$. Assume $a \in (1/4,1/2]$ and set
\begin{equation}\label{p}
  p = \min\bigl(a,3(a-1/4)\bigr).
\end{equation}
Then for all $\sigma \in [0,\sigma_0]$ we have
\begin{align}
  \label{Mest}
  \sup_{t \in [0,\delta]} \mathfrak M_\sigma(t)
  &\le
  \mathfrak M_\sigma(0) + c \delta^{p} \sigma^{1/2-a} \mathfrak M_\sigma(0) \left( \mathfrak M_\sigma(0)^{1/2} + \mathfrak N_\sigma(0) \right),
  \\
  \label{Nest}
  \sup_{t \in [0,\delta]} \mathfrak N_\sigma(t)
  &\le
  \mathfrak N_\sigma(0) + c \delta^{1/2} \norm{\psi(0,\cdot)}_{L^2}^2 + c \delta^{p} \sigma^{1/2-a} \mathfrak M_\sigma(0),
\end{align}
where the constant $c>0$ depends only on $a$ and $M$.
\end{theorem}

We now have all the tools needed to prove the main result, Theorem \ref{Thm1}.

\subsection{Proof of Theorem \ref{Thm1}}\label{MainProof}

Without loss of generality we restrict attention to $t \ge 0$. We must prove the lower bound $\sigma(t) \ge \sigma_0 e^{-At}$ for all $t \ge 0$, for some constant $A > 0$ depending on $\sigma_0$ and the norm of the data. But by Theorem \ref{LWPa} there exists $t_0 > 0$ such that $\sigma(t) \ge \sigma_0$ for all $t \in [0,t_0]$, hence it suffices to show a lower bound $\sigma(t) \ge c e^{-Bt}$ for some constants $c,B > 0$ depending on $\sigma_0$ and the data norm.  We split the proof into two steps.

Fix $a \in (1/4,1/2)$, define $p$ by \eqref{p} and set $q=1/2-a$ and $r = (3/2-p)/q$. Let $c_0$ and $c$ be the constants from Theorems \ref{LWPa} and \ref{ApproxCons}. We will denote by
\[
  K = \norm{\psi_0}_{L^2}^2
\]
the conserved charge \eqref{ChargeCons}. We fix $R_0 \ge 1$ so large that $\sigma_0^q R_0^{3/2-p} \ge 1$, $c c_0^{1/2} K \le R_0$ and $11^{3/2} c c_0^p \le R_0$.

\subsubsection{Step 1} Let $R \ge R_0$ be so large that
\[
  \mathfrak M_{\sigma_0}(0) + \mathfrak N_{\sigma_0}(0)^2 \le R,
\]
and set
\[
  \delta = \frac{c_0}{12R},
\]
where $c_0$ is as in the local existence theorem, Theorem \ref{LWPa}. Iterating that result, with $\sigma_0$ replaced by a parameter $\sigma \in (0,\sigma_0]$, we cover successive time intervals $[0,\delta]$, $[\delta,2\delta]$ etc. In fact, we choose $\sigma$ so that $\sigma^q R^{3/2-p} = 1$, that is,
\[
  \sigma = R^{-r} = R^{-(3/2-p)/q}.
\]
Proceeding inductively, let us assume that for some $n \in \N$ we have
\[
  \sup_{t \in [0,(n-1)\delta]} \left( \mathfrak M_{\sigma}(t) + \mathfrak N_{\sigma}(t)^2 \right) \le 11R.
\]
Then by Theorem \ref{LWPa} (with $\sigma_0$ replaced by $\sigma$) we can extend the solution to $[0,n\delta]$, and by Theorem \ref{ApproxCons},
\begin{align*}
  \sup_{t \in [0,n\delta]} \mathfrak M_\sigma(t)
  &\le
  \mathfrak M_\sigma(0) + nc \delta^p \sigma^q (11R)^{3/2},
  \\
  \sup_{t \in [0,n\delta]} \mathfrak N_\sigma(t)
  &\le
  \mathfrak N_\sigma(0) + ncK \delta^{1/2} + nc \delta^p \sigma^q 11R.
\end{align*}
Since $\mathfrak M_\sigma(0) \le \mathfrak M_{\sigma_0}(0) \le R$ and $\mathfrak N_\sigma(0) \le \mathfrak N_{\sigma_0}(0) \le R^{1/2}$, we then get
\begin{align*}
  \sup_{t \in [0,n\delta]} \mathfrak M_\sigma(t)
  &\le
  R + nc \delta^p \sigma^q (11R)^{3/2},
  \\
  \sup_{t \in [0,n\delta]} \mathfrak N_\sigma(t)
  &\le
  R^{1/2} + ncK\delta^{1/2} + nc \delta^p \sigma^q 11R.
\end{align*}
Thus, if
\begin{equation}\label{Rcond}
  nc \delta^p \sigma^q 11^{3/2} R^{1/2} \le 1
  \quad \text{and} \quad
  ncK\delta^{1/2} \le R^{1/2},
\end{equation}
it follows that
\[
  \sup_{t \in [0,n\delta]} \left( \mathfrak M_{\sigma}(t) + \mathfrak N_{\sigma}(t)^2 \right) \le 11R.
\]
Note that \eqref{Rcond} certainly holds for $n=1$, by the choice $\sigma=R^{-r}$, $R \ge R_0$ and the assumptions on $R_0$.

Setting $T=n\delta$ and using $\delta=c_0/12R$ and $\sigma^q R^{3/2-p} = 1$, we rewrite \eqref{Rcond} as
\[
  T \cdot \max\left( 11^{3/2} c \left(\frac{c_0}{12}\right)^{p-1}, c \left(\frac{c_0}{12}\right)^{-1/2} K \right) \le 1.
\]

The induction stops at time $T=n\delta$, where $n$ is the largest natural number such that \eqref{Rcond} holds. It follows that
\[
  T \ge T_0 := \frac{1}{2\mu}, \quad \text{where} \quad \mu = \max\left( 11^{3/2} c \left(\frac{c_0}{12}\right)^{p-1}, c \left(\frac{c_0}{12}\right)^{-1/2} K \right) > 0.
\]
Indeed, since \eqref{Rcond} fails when $n$ is replaced by $n+1$, we have $1 < (T+\delta)\mu \le 2T\mu$.

To summarize, what we have proved in Step 1 is that there exists $T_0 > 0$, depending only on $a$, $\mu$ and the conserved charge, such that for any $R \ge R_0$ and for any data at $t=0$ satisfying $\mathfrak M_{\sigma_0}(0) + \mathfrak N_{\sigma_0}(0)^2 \le R$, the solution has radius of analyticity at least $\sigma = R^{-r}$ for all $t \in [0,T_0]$, and we have the final-time bound $\mathfrak M_{\sigma}(T_0) + \mathfrak N_{\sigma}(T_0)^2 \le 11R$.

\subsubsection{Step 2}

We iterate the result of Step 1. Proceeding inductively we cover intervals $[(n-1)T_0,nT_0]$ for $n=1,2,\dots$, on each of which the radius of analyticity is at least
\[
  \sigma = \sigma_n = (11^{n-1}R)^{-r}
\]
and we have the final-time bound
\[
  \mathfrak M_{\sigma_n}(nT_0) + \mathfrak N_{\sigma_n}(nT_0)^2 \le 11^n R.
\]
Thus
\[
  \sigma(t) \ge R^{-r} e^{-(\ln 11/T_0)t}
\]
for $t \ge 0$, as desired. This concludes the proof of Theorem \ref{Thm1}.

\section{Function spaces}\label{Spaces}

We impose the convention that the letters $N$ and $L$ (and indexed versions of these) denote elements of the set of dyadic numbers $2^{\N_0} = \left\{ 2^n \colon n \in \N_0 \right\}$, and that sums, unions and supremums over $N$ or $L$ are tacitly understood to be restricted to this set. Define disjoint dyadic sets $S_N$ by
\[
  S_1 = (-1,1), \qquad S_{2^n} = (-2^n,-2^{n-1}] \cup [2^{n-1},2^n) \quad \text{for $n=1,2,\dots$},
\]
so that $\R = \cup_{N} S_N$.

Note that each equation in \eqref{DKG} is of the form $(-i\partial_t + h(D))u=F$, with $h(\xi)=\pm\abs{\xi}$ or $\pm\angles{\xi}$. In general, given a continuous $h \colon \R^2 \to \R$ of polynomial growth we consider the family of norms, for $\sigma \ge 0$, $s,b \in \R$ and $1 \le p < \infty$,
\[
  \norm{u}_{X^{\sigma,s,b;p}_{h(\xi)}}
  =
  \left( \sum_{L} L^{bp} \norm{ e^{\sigma\norm{\xi}} \angles{\xi}^s \chi_{S_L}(\tau+h(\xi))
  \widehat u(\tau,\xi)}_{L^2_{\tau,\xi}}^p \right)^{1/p},
\]
and for $p=\infty$,
\[
  \norm{u}_{X^{\sigma,s,b;\infty}_{h(\xi)}}
  =
  \sup_{L} L^b \norm{ e^{\sigma\norm{\xi}} \angles{\xi}^s \chi_{S_L}(\tau+h(\xi))
  \widehat u(\tau,\xi)}_{L^2_{\tau,\xi}}.
\]
Here $\chi_{S_L}$ denotes the characteristic function of $S_L$ and
\[
  \widehat u(\tau,\xi) = \mathcal F u(\tau,\xi) = \int_{\R \times \R^2} e^{-i(t\tau+x\cdot\xi)} u(t,x) \, dt \, dx \qquad (\tau \in \R,\, \xi \in \R^2)
\]
is the space-time Fourier transform. The above norms are the analytic counterparts of the norms used in \cite{GP2010}, the only difference being that we insert the exponential factor $e^{\sigma\norm{\xi}}$.

\begin{definition}
Let $\sigma \ge 0$, $s,b \in \R$ and $1 \le p \le \infty$. The space $X^{\sigma,s,b;p}_{h(\xi)}$ is the set of $u \in \mathcal S'(\R^{1+2})$ such that $\widehat u \in L^1_{\mathrm{loc}}(\R^{1+2})$ and $\norm{u}_{X^{\sigma,s,b;p}_{h(\xi)}} < \infty$. In the case $\sigma=0$ we simplify the notation to $X^{s,b;p}_{h(\xi)} = X^{0,s,b;p}_{h(\xi)}$.
\end{definition}

We define multipliers $P_N$ and $Q_L = Q_L^{h(\xi)}$ by
\begin{align*}
  \widehat{P_N u}(\tau,\xi) &= \chi_{S_N}(\abs{\xi}) \widehat u(\tau,\xi),
  \\
  \widehat{Q_L u}(\tau,\xi) &= \chi_{S_L}(\tau+h(\xi)) \widehat u(\tau,\xi).
\end{align*}
We also write $Q_{\le L_0} = \sum_{L \le L_0} Q_L$ and $Q_{> L_0} = \text{Id} - Q_{\le L_0}$, and similarly for $P$. For convenience we shall use the shorthand $u_N = P_N u$, $u_L = Q_L u$ and $u_{N,L} = P_N Q_L u$.

It is easy to see that the norms corresponding to $h(\xi)=\pm\abs{\xi}$ and $h(\xi) = \pm\angles{\xi}$ are comparable. The spaces $X^{\sigma,s,b;p}_{\pm\abs{\xi}}$ and $X^{\sigma,s,b;p}_{\pm\angles{\xi}}$ therefore coincide and have equivalent norms, so for our purposes they can be used interchangeably and we will denote either of them by $X^{\sigma,s,b;p}_{\pm}$. We will also write $Q_L^{\pm} = Q_L^{\pm\abs{\xi}}$.

We now discuss the main properties of the above spaces. For this purpose it is just as well to work in the general setting of a given continuous $h \colon \R^2 \to \R$ of polynomial growth, and for the remainder of this section we fix such a function.

\begin{lemma}
$X^{\sigma,s,b;p}_{h(\xi)}$ is a Banach space.
\end{lemma}

\begin{proof}
It suffices to exhibit an isometric isomorphism $u \mapsto g = (g_L)_{L \in 2^{\N_0}}$ from $X^{\sigma,s,b;p}_{h(\xi)}$ onto a closed subspace $\mathcal M$ of  $l^p(2^{\N_0};Y)$, where $Y = L^2(e^{2\sigma\norm{\xi}} \angles{\xi}^{2s} \, d\tau \, d\xi)$. The map is given by $g_L(\tau,\xi) = L^b \widehat{u_L}(\tau,\xi)$, and $\mathcal M$ is the subspace of all $g = (g_L) \in l^p(2^{\N_0};Y)$ such that each $g_L$ is supported in $A_L = \{ (\tau,\xi) \colon \tau+h(\xi) \in S_L\}$.

To prove that the map is onto $\mathcal M$, let $g \in \mathcal M$ and define $U \colon \R^{1+2} \to \C$ by $U(\tau,\xi) = L^{-b} g_L(\tau,\xi)$ for $(\tau,\xi) \in A_L$. By the assumption that $h(\xi)$ has polynomial growth, it is easy to see that $U$ is a tempered function, hence $u = \mathcal F^{-1}U$ is well defined and belongs to $X^{\sigma,s,b;p}_{h(\xi)}$.
\end{proof}

\begin{lemma}\label{Emb}
$X^{\sigma,s,1/2;1}_{h(\xi)}$ embeds continuously into $C(\R;G^{\sigma,s})$.
\end{lemma}

\begin{proof}
This follows from
\[
  \norm{u(t)}_{G^{\sigma,s}}
  \le \sum_L \norm{e^{\sigma\norm{\xi}} \angles{\xi}^s \int \abs{\widehat{u_L}(\tau,\xi)}\, d\tau}_{L_{\xi}^2}
  \le \sum_L L^{1/2} \norm{e^{\sigma\norm{\xi}} \angles{\xi}^s \widehat{u_L}(\tau,\xi)}_{L_{\tau,\xi}^2}.
\]
Similarly, we bound $\norm{u(t+h)-u(t)}_{G^{\sigma,s}}$ by the right-hand side with the factor $e^{i\tau h}-1$ inserted in the $L^2_{\tau,\xi}$ norm, and the resulting expression converges to zero as $h \to 0$, by the dominated convergence theorem.
\end{proof}

\begin{lemma}\label{ConvergenceLemma}
Assume $1 \le p < \infty$ and let $u \in X^{\sigma,s,b;p}_{h(\xi)}$. Then $u = \sum_L u_L$ and $u=\sum_{N,L} u_{N,L}$ hold in $X^{\sigma,s,b;p}_{h(\xi)}$. If $p=\infty$, the convergence holds in $\mathcal S'$.
\end{lemma}

\begin{proof}
The Fourier transforms of $u-\sum_{L \le L_0} Q_L u$ and $u-\sum_{N \le N_0} \sum_{L \le L_0} P_N Q_L u$ equal $\widehat u$ multiplied by the characteristic functions of the regions, respectively, (i) $\abs{\tau+h(\xi)} \ge L_0$ and (ii) $\abs{\tau+h(\xi)} \ge L_0$ or $\abs{\xi} \ge N_0$. If $p < \infty$, the $X^{\sigma,s,b;p}_{h(\xi)}$ norm is therefore bounded by
\[
  \left( \sum_{L \ge L_0} L^{bp} \norm{ e^{\sigma\norm{\xi}} \angles{\xi}^s \widehat{u_L}}_{L^2}^p \right)^{1/p}
  +
  \left( \sum_{L < L_0} L^{bp} \norm{ e^{\sigma\norm{\xi}} \angles{\xi}^s \chi_{\abs{\xi} \ge N_0}\widehat{u_L}}_{L^2}^p \right)^{1/p},
\]
and these terms are arbitrarily small for $L_0$ and $N_0$ large enough, by the dominated convergence theorem. If $p=\infty$, the convergence in $\mathcal S'$ follows from dominated convergence on the Fourier side of the Plancherel identity (see \eqref{DualityPairing} below) when $u$ is tested on any $v \in \mathcal S$.
\end{proof}

We remark that the Schwartz class $\mathcal S(\R^{1+2})$ is contained in $X^{\sigma,s,b;p}_{h(\xi)}$ if $\sigma=0$, but not if $\sigma > 0$. Recall that we simplify the notation to $X^{0,s,b;p}_{h(\xi)} = X^{s,b;p}_{h(\xi)}$ when $\sigma=0$. We now prove some density and duality results for this case.

\begin{lemma}\label{DensityLemma}
$\mathcal S$ is dense in $X^{s,b;p}_{h(\xi)}$ if $1 \le p < \infty$, but not if $p=\infty$.
\end{lemma}

\begin{proof}
If $1 \le p < \infty$ and $u \in X^{s,b;p}_{h(\xi)}$, then by Lemma \ref{ConvergenceLemma}, $v = \sum_{L \le L_0} u_L$ can be made arbitrarily close to $u$ in $X^{s,b;p}_{h(\xi)}$ by choosing $L_0$ large enough. But the index set of $L$ now being finite, $\widehat v$ belongs to $L^2(\angles{\xi}^{2s} \, d\tau \, d\xi)$, in which $\mathcal S$ is dense. Moreover, approximating $\widehat v$ from $\mathcal S$ in $L^2(\angles{\xi}^{2s} \, d\tau \, d\xi)$, one approximates also in $X^{s,b;p}_{h(\xi)}$.

If $p=\infty$, set $\widehat u(\tau,\xi) = \angles{\tau}^{-b-1/2} \chi_{S_1}(\abs{\xi})$. Then $L^b \norm{ \angles{\xi}^s \widehat{u_L}}_{L^2} \sim 1$ for large $L$, so $u \in X^{s,b;\infty}_{h(\xi)}$. Moreover, for any $v \in \mathcal S$ we have $L^b \norm{ \angles{\xi}^s (\widehat{u_L}-\widehat{v_L})}_{L^2} \gtrsim 1$ for large $L$, so approximation from $\mathcal S$ is impossible in $X^{s,b;\infty}_{h(\xi)}$.
\end{proof}

A duality pairing between $X^{s,b;p}_{h(\xi)}$ and $X^{-s,-b;p'}_{h(\xi)}$ can be defined in a natural way as an extension of the pairing sending $(u,v) \in \mathcal S' \times \mathcal S$ to
\begin{equation}\label{DualityPairing}
  \int u\overline v \, dt \, dx
  =
  (2\pi)^{-3} \int \widehat{u}(\tau,\xi) \overline{\widehat{v}(\tau,\xi)} \, d\tau \, d\xi,
\end{equation}
where equality holds by Plancherel's theorem. But the right side is well defined as an absolutely convergent integral for any $(u,v) \in X^{s,b;p}_{h(\xi)} \times X^{-s,-b;p'}_{h(\xi)}$, since by Cauchy-Schwarz and H\"older we can bound in absolute value by
\begin{equation}\label{DualityBound}
  \sum_L \norm{\angles{\xi}^s L^b \widehat{u_L}(\tau,\xi)}_{L^2} \norm{\angles{\xi}^{-s} L^{-b} \widehat{v_L}(\tau,\xi)}_{L^2}
  \le
  \norm{u}_{X^{s,b;p}_{h(\xi)}} \norm{v}_{X^{-s,-b;p'}_{h(\xi)}}.
\end{equation}
For $(u,v) \in X^{s,b;p}_{h(\xi)} \times X^{-s,-b;p'}_{h(\xi)}$ we can therefore consistently define $\int u\overline v \, dt \, dx$ by \eqref{DualityPairing}. This bilinear pairing is bounded, and hence continuous, in view of \eqref{DualityBound}. With this definition, we have the following.

\begin{lemma}\label{DualityLemma} Let $1 \le p \le \infty$. For any $u \in X^{s,b;p}_{h(\xi)}$ we have
\begin{equation}\label{duality}
  \norm{u}_{X^{s,b;p}_{h(\xi)}}
  =
  \sup \left\{ (2\pi)^{3} \Abs{\int u\overline v \, dt \, dx} \colon v \in X^{-s,-b;p'}_{h(\xi)}, \; \norm{v}_{X^{-s,-b;p'}_{h(\xi)}} = 1 \right\},
\end{equation}
where $1 \le p' \le \infty$ is the H\"older conjugate of $p$, defined by $\frac{1}{p}+\frac{1}{p'}=1$. Moreover, the set over which the supremum is taken can be further restricted as follows:
\begin{enumerate}
\item\label{duality1} if $p > 1$, we can restrict to $v \in \mathcal S$;
\item\label{duality2} if $p < \infty$, we can restrict to $v$ such that $\widehat v \in L^2$ with compact support.
\end{enumerate}
\end{lemma}

\begin{proof}
By \eqref{DualityPairing} and \eqref{DualityBound}, LHS\eqref{duality} $\ge$ RHS\eqref{duality}. Conversely, if $1 \le p < \infty$, then defining $v$ by
\begin{equation}\label{DualFunction}
  \widehat{v_L}(\tau,\xi) = \frac{\angles{\xi}^{2s} L^{2b} \widehat{u_L}(\tau,\xi)}{\norm{\angles{\cdot}^{s} L^{b} \widehat{u_L}}_{L^2}^{2-p} \norm{u}_{X^{s,b;p}_{h(\xi)}}^{p/p'}}
\end{equation}
for all $L$ for which $\norm{\angles{\cdot}^{s} L^{b} \widehat{u_L}}_{L^2} > 0$, and $\widehat{v_L}(\tau,\xi) = 0$ for all other $L$, we have $\norm{v}_{X^{-s,-b;p'}_{h(\xi)}} = 1$ (we assume of course that LHS\eqref{duality} is not zero) and equality holds in \eqref{duality}. If $p=\infty$, then fixing $L$ and defining $v$ by
\[
  \widehat v(\tau,\xi) = \frac{\angles{\xi}^{2s} L^{2b} \widehat{u_L}(\tau,\xi)}{\norm{\angles{\cdot}^{s} L^{b} \widehat{u_L}}_{L^2}}
\]
we have $\norm{v}_{X^{-s,-b;1}_{h(\xi)}} = 1$ and $(2\pi)^{3} \int u\overline v \, dt \, dx = \norm{\angles{\cdot}^{s} L^{b} \widehat{u_L}}_{L^2}$. It follows that RHS\eqref{duality} $\ge \norm{\angles{\cdot}^{s} L^{b} \widehat{u_L}}_{L^2}$ for all $L$, hence RHS\eqref{duality} $\ge \norm{u}_{X^{s,b;\infty}_{h(\xi)}}$. This concludes the proof of \eqref{duality}. The claim \eqref{duality1} follows since $\mathcal S$ is dense in $X^{-s,-b;p'}_{h(\xi)}$ for $p' < \infty$. Finally, to prove \eqref{duality2} we assume $p < \infty$ and note that by Lemma \ref{ConvergenceLemma} we can reduce to the case where $\widehat u$ has compact support, hence $\widehat v$ given by \eqref{DualFunction} also has compact support. Moreover, $\widehat v \in L^2$.
\end{proof}

The restriction of $X^{\sigma,s,b;p}_{h(\xi)}$ to a time interval $(-\delta,\delta)$ is denoted $X^{\sigma,s,b;p}_{h(\xi)}(\delta)$. It can be viewed as the quotient space $X^{\sigma,s,b;p}_{h(\xi)} / \mathcal M$, where $\mathcal M$ is the closed subspace consisting of those $u \in X^{\sigma,s,b;p}_{h(\xi)}$ which vanish on $(-\delta,\delta) \times \R^2$. The norm
\begin{equation}\label{RestrictionNorm}
  \norm{u}_{X^{\sigma,s,b;p}_{h(\xi)}(\delta)}
  =
  \inf \left\{\norm{v}_{X^{\sigma,s,b;p}_{h(\xi)}} \colon v \in X^{\sigma,s,b;p}_{h(\xi)}, \; \text{$u=v$ on $(-\delta,\delta) \times \R^2$} \right\}
\end{equation}
makes $X^{\sigma,s,b;p}_{h(\xi)}(\delta)$ a Banach space. As before, we write $X^{0,s,b;p}_{h(\xi)}(\delta) = X^{s,b;p}_{h(\xi)}(\delta)$.

\begin{lemma}\label{CutoffLemma}
Let $\sigma \ge 0$, $s \in \R$, $0 < b \le 1/2$ and $0 < \delta \le 1$. Then
\[
  \norm{u}_{X^{\sigma,s,b;1}_{h(\xi)}(\delta)}
  \le c \delta^{1/2-b} \norm{u}_{X^{\sigma,s,1/2;1}_{h(\xi)}(\delta)},
\]
where $c$ depends only on $b$.
\end{lemma}

\begin{proof}
Replacing $u$ by $e^{\sigma\norm{D}}\angles{D}^su$ we reduce to the case $\sigma=s=0$, which is proved in \cite[Proposition 2.1(iii)]{GP2010}.
\end{proof}

\begin{lemma}\label{SharpTimeRestrictionLemma}
Let $s \in \R$, $-1/2 < b < 1/2$, $1 \le p < \infty$ and $0 < \delta \le 1$. Then for any time interval $I \subset [-\delta,\delta]$ we have the estimate
\begin{equation}\label{Restricted}
  \norm{\chi_{I} u}_{X^{s,b;p}_{h(\xi)}} \le c \norm{u}_{X^{s,b;p}_{h(\xi)}(\delta)},
\end{equation}
where $\chi_I(t)$ is the characteristic function of $I$, and $c$ depends only on $b$.
\end{lemma}

\begin{proof} In view of the definition \eqref{RestrictionNorm} of the restriction norm, it suffices to prove
\[
  \norm{\chi_{I} u}_{X^{s,b;p}_{h(\xi)}} \le c \norm{u}_{X^{s,b;p}_{h(\xi)}}.
\]
We adapt an argument from \cite[Lemma 3.2]{CKSTT2004}. Since $p < \infty$, $\mathcal S$ is dense in $X^{s,b;p}_{h(\xi)}$ by Lemma \ref{DensityLemma}, so it is enough to prove the estimate for $u \in \mathcal S$. Replacing $u$ by $\angles{D}^s u$, we may assume $s=0$. Writing $\chi_I(t)$ in terms of signum functions and applying Lemma \ref{DualityLemma}, we then reduce to proving
\begin{equation}\label{SharpTimeRestrictionLemma:A}
  \Abs{ \int \sgn(t) u(t,x) \overline{v(t,x)} \, dt \, dx } \le c \norm{u}_{X^{0,b;p}_{h(\xi)}} \norm{v}_{X^{0,-b;p'}_{h(\xi)}}
\end{equation}
for $u \in \mathcal S$ and $v \in X^{0,-b;p'}_{h(\xi)} \cap L^2$. We bound the left side by
\[
  \sum_{L_1,L_2} \Abs{ \int \sgn(t) u_{L_1}(t,x) \overline{v_{L_2}(t,x)} \, dt \, dx }
\]
and separate the cases $L_1 \sim L_2$, $L_1 \ll L_2$ and $L_2 \ll L_1$. For $L_1 \sim L_2$ we bound by
\[
  \sum_{L_1 \sim L_2} \norm{u_{L_1}}_{L^2} \norm{v_{L_2}}_{L^2}
  \sim
  \sum_{L_1 \sim L_2} L_1^b \norm{u_{L_1}}_{L^2} L_2^{-b} \norm{v_{L_2}}_{L^2}
  \lesssim
  \norm{u}_{X^{0,b;p}_{h(\xi)}} \norm{v}_{X^{0,-b;p'}_{h(\xi)}},
\]
while for $L_1 \ll L_2$ we write
\begin{multline*}
  \int \sgn(t) u_{L_1}(t,x) \overline{v_{L_2}(t,x)} \, dt \, dx
  =
  \lim_{n \to \infty} \int \phi\left(\frac{t}{n}\right) u_{L_1}(t,x) \overline{v_{L_2}(t,x)} \, dt \, dx
  \\
  =
  c
  \lim_{n \to \infty}
  \int
  n\widehat \phi(n[\tau-\lambda])
  \widehat{u_{L_1}}(\lambda,\xi)
  \overline{\widehat{v_{L_2}}(\tau,\xi)} \, d\lambda \, d\tau \, d\xi,
\end{multline*}
where $\phi(t) = \sgn(t) \chi_{[-1,1]}(t)$ has Fourier transform $\widehat \phi(\tau) = O(\abs{\tau}^{-1})$ and
\[
  \abs{\tau-\lambda} = \abs{(\tau+h(\xi)) - (\lambda+h(\xi))} \sim L_2,
\]
hence we dominate in this case by
\begin{align*}
  &\sum_{L_1 \ll L_2}
  L_2^{-1} \int \abs{\widehat{u_{L_1}}(\lambda,\xi)} \abs{\widehat{v_{L_2}}(\tau,\xi)} \, d\lambda \, d\tau \, d\xi
  \\
  &\qquad\le c
  \sum_{L_1 \ll L_2}
  \left(\frac{L_1}{L_2}\right)^{1/2} \int \norm{\widehat{u_{L_1}}(\lambda,\xi)}_{L^2_\lambda} \norm{\widehat{v_{L_2}}(\tau,\xi)}_{L^2_\tau} \, d\xi
  \\
  &\qquad\le c
  \sum_{L_1 \ll L_2}
  \left(\frac{L_1}{L_2}\right)^{1/2-b} L_1^b \norm{\widehat{u_{L_1}}}_{L^2} L_2^{-b}\norm{\widehat{v_{L_2}}}_{L^2}
  \\
  &\qquad\le c \sum_{l=0}^\infty \sum_{j=0}^\infty 2^{-l\varepsilon} \alpha_j \beta_{j+l}
  \le c \norm{\alpha}_{l^p} \norm{\beta}_{l^{p'}}
  = c \norm{u}_{X^{0,b;p}_{h(\xi)}} \norm{v}_{X^{0,-b;p'}_{h(\xi)}},
\end{align*}
where $\varepsilon=1/2-b>0$, $L_1=2^j$, $L_2=2^{j+l}$, $\alpha_j = L_1^b\norm{\widehat{u_{L_1}}}_{L^2}$ and $\beta_{j+l} = L_2^{-b}\norm{\widehat{v_{L_2}}}_{L^2}$. Here we used $b < 1/2$. The remaining case $L_2 \ll L_1$ works out similarly, but relies on $-b < 1/2$.
\end{proof}

In terms of the free propagator $U(t) = e^{-ith(D)}$ the solution of
\begin{equation}\label{IVP}
  \left(-i\partial_t + h(D)\right) u = F,
  \qquad
  u(0,x) = f(x),
\end{equation}
is given, for sufficiently regular $F(t,x)$ and $f(x)$, by Duhamel's formula
\begin{equation}\label{Duhamel}
  u(t) = U(t) f + i\int_0^t U(t-t')F(t') \, dt',
\end{equation}
and satisfies the following estimate.

\begin{lemma}\label{IVPlemma} Let $\sigma \ge 0$, $s \in \R$, $-1/2 < b < 1/2$ and $0 < \delta \le 1$. For any $f \in G^{\sigma,s}$ and $F \in X^{\sigma,s,b;\infty}_{h(\xi)}(\delta)$ there is a unique $u \in X^{\sigma,s,1/2;1}_{h(\xi)}(\delta)$ satisfying the initial value problem \eqref{IVP} on $(-\delta,\delta) \times \R^2$. Moreover,
\begin{equation}\label{CutoffEstimate}
  \norm{u}_{X^{\sigma,s,1/2;1}_{h(\xi)}(\delta)}
  \le c \left( \norm{f}_{G^{\sigma,s}} + \delta^{1/2+b} \norm{F}_{X^{\sigma,s,b;\infty}_{h(\xi)}(\delta)} \right),
\end{equation}
where $c$ depends only on $b$.
\end{lemma}

\begin{proof}
By the substitution $u \to e^{\sigma\norm{D}}\angles{D}^s u$ we reduce to the case $\sigma=s=0$. The proof now follows more or less along the lines of the proof of the analogous result for the standard $X^{s,b} = X^{s,b;2}$ spaces, but some care must be taken since $\mathcal S$ is not dense in $X^{0,b;\infty}_{h(\xi)}$. Assuming for the moment $F \in \mathcal S$, then \eqref{Duhamel} can be rewritten, via the Fourier transform, as
\[
  u(t) = U(t) f
  + (TF)(t),
\]
where
\begin{align*}
  (TF)(t) &= \sum_{n=1}^\infty \frac{t^n}{n!} U(t) f_n
  + U(t) g
  + \mathcal F^{-1} \left( \frac{\mathcal F\{Q_{L>\delta^{-1}}F\}(\tau,\xi)}{\tau+h(\xi)} \right),
  \\
  \widehat{f_n}(\xi) &= c \int (\tau+h(\xi))^{n-1} \mathcal F\{Q_{L \le \delta^{-1}} F\}(\tau,\xi) \, d\tau,
  \\
  \widehat g(\xi) &= c \int (\tau+h(\xi))^{-1} \mathcal F\{Q_{L > \delta^{-1}} F\}(\tau,\xi) \, d\tau.
\end{align*}
Now one observes that $TF$ is well-defined for any $F \in X^{0,b;\infty}_{h(\xi)}$ and that \eqref{CutoffEstimate} holds; see \cite[Section 13.2]{DS2011}. However, it is not obvious that $TF$ then satisfies \eqref{IVP} with $f=0$. But choosing $b' \in (-1/2,b)$ we have $F \in X^{0,b;\infty}_{h(\xi)} \subset X^{0,b';2}_{h(\xi)}$. In the latter space, $\mathcal S$ is dense, and by a well-known result the linear operator $T$ is bounded from $X^{0,b';2}_{h(\xi)}(\delta)$ into $X^{0,b'+1;2}_{h(\xi)}(\delta)$ and $TF$ satisfies \eqref{IVP} on $(-\delta,\delta) \times \R^2$ with $f=0$.
\end{proof}

\begin{corollary}\label{IVPcor}
Under the assumptions of Lemma \ref{IVPlemma} we have
\[
  \sup_{t \in [-\delta,\delta]} \norm{u(t)}_{G^{\sigma,s}} \le \norm{f}_{G^{\sigma,s}} + c\delta^{1/2+b} \norm{F}_{X^{\sigma,s,b;\infty}_{h(\xi)}(\delta)}.
\]
\end{corollary}

\begin{proof}
For the first term in \eqref{Duhamel} we use $\norm{U(t)f}_{G^{\sigma,s}} \le \norm{f}_{G^{\sigma,s}}$, and for the second term we use Lemma \ref{Emb} and Lemma \ref{IVPlemma}.
\end{proof}

\section{Multilinear space-time estimates}\label{Spacetime}

Estimating the solution of \eqref{DKG} via duality (Lemma \ref{DualityLemma}), the need arises for the following trilinear space-time estimates, which we shall prove by combining dyadic bilinear $L^2$ space-times estimates (stated in Lemma \ref{BilinearLemma} below) with the null form estimate \eqref{NullEstimate}. The special case $\sigma=0$, $a=1/2$ and $b_0=b_1=b_2=1/3$ of the following theorem was proved in \cite{GP2010}.

\begin{theorem}\label{TrilinearTheorem} Assume that
\begin{itemize}
\item $a \in (1/4,3/4]$,
\item $b_0,b_1,b_2 \ge \max(1/4,3/4-a)$
\item $b_0+b_1+b_2 \ge 3/2-a$.
\end{itemize}
Then there exists a constant $c > 0$ such that the following estimates hold for all signs $s_0,s_1,s_2 \in \{+,-\}$ and for all $\sigma \ge 0$:
\begin{align}
  \label{Trilinear1}
  \Abs{ \int_{\R^{1+2}} \left(e^{\sigma\norm{D}} \phi \right)\innerprod{\beta\Pi_{s_1}\psi_1}{\Pi_{s_2}\psi_2} \, dt \, dx }
  &\le c
  \norm{\phi}_{X^{0,a,b_0;1}_{s_0}}
  \norm{\psi_1}_{X^{\sigma,0,b_1;1}_{s_1}}
  \norm{\psi_2}_{X^{\sigma,0,b_2;1}_{s_2}},
  \\
  \label{Trilinear2}
  \Abs{ \int_{\R^{1+2}} \phi \innerprod{\beta\Pi_{s_1}\psi_1}{\Pi_{s_2}e^{\sigma\norm{D}}\psi_2} \, dt \, dx }
  &\le c
  \norm{\phi}_{X^{\sigma,a,b_0;1}_{s_0}}
  \norm{\psi_1}_{X^{\sigma,0,b_1;1}_{s_1}}
  \norm{\psi_2}_{X^{0,0,b_2;1}_{s_2}}.
\end{align}
\end{theorem}

The proof is given at the end of this section. Before proceeding we record the following consequence of Theorem \ref{TrilinearTheorem}.

\begin{corollary}\label{BilinearCor}
Under the assumptions of Theorem \ref{TrilinearTheorem} there exists $c > 0$ such that for all $\sigma \ge 0$, $\delta \in (0,1]$ and signs $s_0,s_1,s_2 \in \{+,-\}$ we have the estimates
\begin{align*}
  \norm{\innerprod{\beta\Pi_{s_1}\psi_1}{\Pi_{s_2}\psi_2}}_{X_{s_0}^{\sigma,-a,-b_0;\infty}(\delta)}
  &\le c \delta^{1-b_1-b_2}
  \norm{\psi_1}_{X^{\sigma,0,1/2;1}_{s_1}(\delta)}
  \norm{\psi_2}_{X^{\sigma,0,1/2;1}_{s_2}(\delta)},
  \\
  \norm{\Pi_{s_2}\left(\phi \beta\Pi_{s_1}\psi_1\right)}_{X_{s_2}^{\sigma,0,-b_2;\infty}(\delta)}
  &\le c \delta^{1-b_0-b_1}
  \norm{\phi}_{X^{\sigma,a,1/2;1}_{s_0}(\delta)}
  \norm{\psi_1}_{X^{\sigma,0,1/2;1}_{s_1}(\delta)}.
\end{align*}
\end{corollary}

\begin{proof}
We only give the details for the first estimate. By Lemma \ref{CutoffLemma} we reduce to
\[
  \norm{\innerprod{\beta\Pi_{s_1}\psi_1}{\Pi_{s_2}\psi_2}}_{X_{s_0}^{\sigma,-a,-b_0;\infty}(\delta)}
  \le c
  \norm{\psi_1}_{X^{\sigma,0,b_1;1}_{s_1}(\delta)}
  \norm{\psi_2}_{X^{\sigma,0,b_2;1}_{s_2}(\delta)}.
\]
Working with extensions, we note that it suffices to prove the estimate without the restriction to the time interval $(-\delta,\delta)$. Thus we need to prove
\[
  \norm{\innerprod{\beta\Pi_{s_1}\psi_1}{\Pi_{s_2}\psi_2}}_{X_{s_0}^{\sigma,-a,-b_0;\infty}}
  \le c
  \norm{\psi_1}_{X^{\sigma,0,b_1;1}_{s_1}}
  \norm{\psi_2}_{X^{\sigma,0,b_2;1}_{s_2}},
\]
but this follows from Theorem \ref{TrilinearTheorem} via Lemma \ref{DualityLemma}.
\end{proof}

There is no $L^4$ space-time estimate for free solutions of the wave equation in two space dimensions, and hence no $L^2$ product estimate. As observed in \cite{Selberg2011}, one can nevertheless prove Fourier restriction estimates on truncated thickened null cones in space-time, such as the ones in the following lemma, which will be used to prove Theorem \ref{TrilinearTheorem}.

Some notation: Given dyadic numbers $N_0,N_1,N_2,L_0,L_1,L_2 \ge 1$, we denote by $\Lmin$, $\Lmed$ and $\Lmax$ the minimum, median and maximum of $L_0$, $L_1$ and $L_2$, and similarly for the $N$'s. Moreover, for $j,k \in \{0,1,2\}$, $j < k$, we denote by $\Lmin^{jk}$ (resp.\ $\Lmax^{jk}$) the minimum (resp.\ the maximum) of $L_j$ and $L_k$, and similarly for the $N$'s. We also write $\mathbf N = (N_0,N_1,N_2)$ and $\mathbf L = (L_0,L_1,L_2)$. We will use the notation $N \lesssim N'$, $N \ll N'$ and $N \sim N'$ as shorthand for, respectively, $N \le cN'$, $N \le c^{-1}N'$ and $c^{-1}N' \le N \le cN'$, where $c$ is a sufficiently large absolute constant. From now on we use the notation $Q_L^{\pm}$ for the modulation operator $Q_L^{\pm\abs{\xi}}$ defined in the previous subsection (note that we could also have used $Q_L^{\pm\angles{\xi}}$).

\begin{lemma}\label{BilinearLemma}
There exists $c > 0$ such that for all dyadic numbers $N_j,L_j \ge 1$, $j,k \in \{0,1,2\}$, and all signs $s_0,s_1,s_2 \in \{+,-\}$ we have the bilinear $L^2$ space-time estimate
\[
  \norm{P_{N_0} Q_{L_0}^{s_0} \left( P_{N_1} Q_{L_1}^{s_1} u_1 \cdot P_{N_2} Q_{L_2}^{s_2} u_2 \right)}_{L^2(\R^{1+2})}
  \le C(\mathbf N,\mathbf L)
  \norm{u_1}_{L^2(\R^{1+2})}
  \norm{u_2}_{L^2(\R^{1+2})},
\]
where
\begin{multline*}
  C(\mathbf N,\mathbf L)
  = c \min\left[
  \bigl( (\Nmin)^2 \Lmin \bigr)^{1/2},
  \bigl( \Nmin\Lmin^{12} \bigr)^{1/2} \bigl( \Nmin^{12}\Lmax^{12} \bigr)^{1/4},
  \right.
  \\
  \bigl( \Nmin\Lmin^{01} \bigr)^{1/2} \bigl( \Nmin^{01}\Lmax^{01} \bigr)^{1/4},
  \left.\bigl( \Nmin\Lmin^{02} \bigr)^{1/2} \bigl( \Nmin^{02}\Lmax^{02} \bigr)^{1/4}
  \right].
\end{multline*}
Moreover, in the case $s_1=s_2$ and $N_0 \ll N_1 \sim N_2$, the above estimate holds also with $C(\mathbf N,\mathbf L) = c (N_0L_1L_2)^{1/2}$.
\end{lemma}

\begin{proof}
The estimate is proved in \cite[Theorem 2.1]{Selberg2011}, except for the last statement about the special case $s_1=s_2$ and $N_0 \ll N_1 \sim N_2$, which is included in \cite[Proposition 9.1, Eq.~(66)]{Tao2001} or alternatively can be deduced from the free-wave estimate in \cite[Theorem 12.1, Eq.~(65)]{FoKl2000} via the transfer principle (by observing that the multiplier $D_-$ is of size $\lambda$, in the notation of that paper).
\end{proof}

\begin{remark}\label{BilinearRemark}
By Plancherel's theorem, the estimate in Lemma \ref{BilinearLemma} is equivalent to
\[
  \norm{I(\tau,\xi)}_{L^2_{\tau,\xi}}
  \le C(\mathbf N,\mathbf L)
  \norm{u_1}_{L^2(\R^{1+2})}
  \norm{u_2}_{L^2(\R^{1+2})},
\]
where
\begin{multline*}
  I(\tau,\xi) = \chi_{S_{N_0}}(\abs{\xi})
  \chi_{S_{L_0}}(\tau+s_0\abs{\xi_0})
  \\
  \times \int
  \Abs{\widehat{P_{N_1}Q^{s_1}_{L_1}u_1}(\tau-\lambda,\xi-\eta)}
  \Abs{\widehat{P_{N_2}Q^{s_2}_{L_2}u_2}(\lambda,\eta)} \, d\lambda \, d\eta,
\end{multline*}
and it is in this form that we will now apply the estimate.
\end{remark}

We are now in a position to prove the trilinear estimates.

\subsection{Proof of Theorem \ref{TrilinearTheorem}}

Using Plancherel's theorem, the self-adjointness of $\Pi(\xi)$, the sign-reversing identity \eqref{Commutation} and the null estimate \eqref{NullEstimate}, we bound the left side of \eqref{Trilinear1} by
\begin{align*}
  &\Abs{ \int e^{\sigma\norm{\xi}} \widehat\phi(\tau,\xi)
  \Innerprod{\beta\Pi\left(s_1(\eta-\xi)\right)\widehat{\psi_1}(\lambda-\tau,\eta-\xi)}
  {\Pi\left(s_2\eta\right)\widehat{\psi_2}(\lambda,\eta)} \, d\lambda \, d\tau \, d\eta \, d\xi }
  \\
  &
  = \Abs{ \int e^{\sigma\norm{\xi}}\widehat\phi(\tau,\xi)
  \Innerprod{\Pi\left(s_2\eta\right)\beta\Pi\left(s_1(\eta-\xi)\right)\widehat{\psi_1}(\lambda-\tau,\eta-\xi)}
  {\widehat{\psi_2}(\lambda,\eta)} \, d\lambda \, d\tau \, d\eta \, d\xi }
  \\
  &
  = \Abs{ \int e^{\sigma\norm{\xi}} \widehat\phi(\tau,\xi)
  \Innerprod{\beta\Pi\left(-s_2\eta\right)\Pi\left(s_1(\eta-\xi)\right)\widehat{\psi_1}(\lambda-\tau,\eta-\xi)}
  {\widehat{\psi_2}(\lambda,\eta)} \, d\lambda \, d\tau \, d\eta \, d\xi }
  \\
  &
  \le c \int \theta_{12}
  \Abs{\widehat\phi(\tau,\xi)}
  e^{\sigma\norm{\eta-\xi}}\Abs{\widehat{\psi_1}(\lambda-\tau,\eta-\xi)}
  e^{\sigma\norm{\eta}}\Abs{\widehat{\psi_2}(\lambda,\eta)} \, d\lambda \, d\tau \, d\eta \, d\xi,
\end{align*}
where
\begin{equation}\label{Angle}
  \theta_{12} = \angle\left(s_1(\eta-\xi),s_2\eta\right)
\end{equation}
and we used the triangle inequality to write
\[
  e^{\sigma\norm{\xi}} \le e^{\sigma\norm{\eta-\xi}}e^{\sigma\norm{\eta}}.
\]
Similarly, the left side of \eqref{Trilinear2} can be bounded by
\[
  c \int \theta_{12}
  e^{\sigma\norm{\xi}}\Abs{\widehat\phi(\tau,\xi)}
  e^{\sigma\norm{\eta-\xi}}\Abs{\widehat{\psi_1}(\lambda-\tau,\eta-\xi)}
  \Abs{\widehat{\psi_2}(\lambda,\eta)} \, d\lambda \, d\tau \, d\eta \, d\xi.
\]
Thus both \eqref{Trilinear1} and \eqref{Trilinear2} reduce to the estimate (without $\sigma$)
\begin{multline}\label{FourierEst}
  \int \theta_{12}
  \Abs{\widehat\phi(\tau,\xi)}
  \Abs{\widehat{\psi_1}(\lambda-\tau,\eta-\xi)}
  \Abs{\widehat{\psi_2}(\lambda,\eta)} \, d\lambda \, d\tau \, d\eta \, d\xi
  \\
  \le c \norm{\phi}_{X^{a,b_0;1}_{s_0}}
  \norm{\psi_1}_{X^{0,b_1;1}_{s_1}}
  \norm{\psi_2}_{X^{0,b_2;1}_{s_2}},
\end{multline}
which we now prove.

By dyadic decomposition we bound the left side by a constant times
\[
  \sum_{\mathbf N, \mathbf L}
  \int \theta_{12}
  \Abs{\widehat{P_{N_0}Q^{s_0}_{L_0}}\phi(\tau,\xi)}
  \Abs{\widehat{P_{N_1}Q^{s_1}_{L_1}\psi_1}(\lambda-\tau,\eta-\xi)}
  \Abs{\widehat{P_{N_2}Q^{s_2}_{L_2}\psi_2}(\lambda,\eta)} \, d\lambda \, d\tau \, d\eta \, d\xi,
\]
where the sum is over dyadic $N_j,L_j \ge 1$, $j = 0,1,2$. The integral vanishes unless the two largest $N$'s are comparable, so we reduce to the cases (i) $N_0 \ll N_1 \sim N_2$, (ii) $N_1 \lesssim N_0 \sim N_2$ or (iii) $N_2 \lesssim N_0 \sim N_1$. By symmetry, it suffices to consider cases (i) and (ii).

To estimate the integral we will apply Cauchy-Schwarz with respect to $(\tau,\xi)$ followed by Lemma \ref{BilinearLemma} with $\widehat{u_1}(\tau-\lambda,\xi-\eta)=\abs{\widehat{\psi_1}(\lambda-\tau,\eta-\xi)}$ and $\widehat{u_2}(\lambda,\eta)=\abs{\widehat{\psi_2}(\lambda,\eta)}$, cp.~Remark \ref{BilinearRemark}. It should be kept in mind that due to the sign change in the argument of $\widehat{u_1}$, the sign $s_1$ is reversed when we apply Lemma \ref{BilinearLemma}.

By \cite[Lemma 2.2]{Selberg2011},
\begin{equation}\label{AngleEst}
  \theta_{12} \le c \left( \frac{\Lmax}{\Nmin^{12}} \right)^{1/2},
\end{equation}
so applying Cauchy-Schwarz and Lemma \ref{BilinearLemma} we bound by a constant times
\[
  S = \sum_{\mathbf N, \mathbf L}
  \left( \min\left(1,\frac{\Lmax}{\Nmin^{12}}\right) \right)^{1/2}
  \frac{C(\mathbf N,\mathbf L)}{N_0^a L_0^{b_0} L_1^{b_1} L_2^{b_2}}
  \alpha_{N_0,L_0} \beta_{N_1,L_1} \gamma_{N_2,L_2},
\]
where
\begin{align*}
  \alpha_{N_0,L_0} &= N_0^a L_0^{b_0} \norm{P_{N_0}Q^{s_0}_{L_0}\phi}_{L^2(\R^{1+2})},
  \\
  \beta_{N_1,L_1} &= L_1^{b_1} \norm{P_{N_1}Q^{s_1}_{L_1}\psi_1}_{L^2(\R^{1+2})},
  \\
  \gamma_{N_2,L_2} &= L_2^{b_2} \norm{P_{N_2}Q^{s_2}_{L_2}\psi_2}_{L^2(\R^{1+2})},
\end{align*}
and $C(\mathbf N,\mathbf L)$ is as in Lemma \ref{BilinearLemma}. It remains to prove that
\begin{equation}\label{Sest}
  S \le c \sum_{\mathbf L} \left( \sum_{N_0} \alpha_{N_0,L_0}^2 \right)^{1/2}
  \left( \sum_{N_1} \beta_{N_1,L_1}^2 \right)^{1/2}
  \left( \sum_{N_2} \gamma_{N_2,L_2}^2 \right)^{1/2}.
\end{equation}

\subsubsection{Case (ii), $N_1 \lesssim N_0 \sim N_2$}

Then $C(\mathbf N,\mathbf L) \le c N_1^{1/2} N_0^{1/4} \Lmin^{1/2} \Lmed^{1/4}$, hence we bound the corresponding part of the sum $S$ by a constant times
\[
  \sum_{\mathbf N, \mathbf L} \mathbb{1}_{N_1 \lesssim N_0 \sim N_2}
  \left( \frac{\Lmax}{N_1} \right)^{\mu}
  \frac{N_1^{1/2} N_0^{1/4} \Lmin^{1/2} \Lmed^{1/4}}{N_0^a L_0^{b_0} L_1^{b_1} L_2^{b_2}}
  \alpha_{N_0,L_0} \beta_{N_1,L_1} \gamma_{N_2,L_2}
\]
for any $\mu \in [0,1/2]$. Clearly
\begin{equation}\label{Lbound}
  \frac{\Lmin^{1/2} \Lmed^{1/4} \Lmax^\mu}{L_0^{b_0} L_1^{b_1} L_2^{b_2}} \le 1
\end{equation}
provided that
\begin{gather}
  \label{cond1}
  b_0+b_1+b_2 \ge \frac34 + \mu,
  \\
  \label{cond2}
  b_0,b_1,b_2 \ge \max\left(\frac14,\mu\right).
\end{gather}
Then we are left with
\[
  \sum_{\mathbf N, \mathbf L} \mathbb{1}_{N_1 \lesssim N_0 \sim N_2}
  \frac{N_1^{1/2-\mu}}{N_0^{a-1/4}}
  \alpha_{N_0,L_0} \beta_{N_1,L_1} \gamma_{N_2,L_2}.
\]
Assuming
\begin{equation}\label{cond3}
  0 \le \mu < \frac12
\end{equation}
we sum $N_1$ and bound by
\[
  \sum_{\mathbf L} \left( \sup_{N_1} \beta_{N_1,L_1} \right) 
  \sum_{N_0,N_2} \mathbb{1}_{N_0 \sim N_2}\frac{N_0^{1/2-\mu}}{N_0^{a-1/4}}
  \alpha_{N_0,L_0} \gamma_{N_2,L_2},
\]
so if $a + \mu - 3/4 \ge 0$, we can sum $N_1 \sim N_2$ by Cauchy-Schwarz to obtain the desired estimate \eqref{Sest}. We therefore choose $\mu = 3/4-a$. Then the conditions \eqref{cond1}, \eqref{cond2} and \eqref{cond3} correspond exactly to the assumptions of the lemma. This concludes the proof in case (ii).

\subsubsection{Case (i), $N_0 \ll N_1 \sim N_2$}

First, if $\Lmax = L_1$ or $\Lmax = L_2$, then Lemma \ref{BilinearLemma} gives $C(\mathbf N,\mathbf L) \le c N_0^{3/4} \Lmin^{1/2} \Lmed^{1/4}$, hence we bound the corresponding part of $S$ by a constant times
\begin{equation}\label{Bound1}
  \sum_{\mathbf N, \mathbf L} \mathbb{1}_{N_0 \ll N_1 \sim N_2}
  \left( \frac{\Lmax}{N_1} \right)^{\mu}
  \frac{N_0^{3/4} \Lmin^{1/2} \Lmed^{1/4}}{N_0^a L_0^{b_0} L_1^{b_1} L_2^{b_2}}
  \alpha_{N_0,L_0} \beta_{N_1,L_1} \gamma_{N_2,L_2}.
\end{equation}
Taking $\mu=3/4-a$ as above, we apply \eqref{Lbound} and reduce to
\[
  \sum_{\mathbf N, \mathbf L} \mathbb{1}_{N_0 \ll N_1 \sim N_2}
  \left(\frac{N_0}{N_1}\right)^{3/4-a}
  \alpha_{N_0,L_0} \beta_{N_1,L_1} \gamma_{N_2,L_2},
\]
so if $a < 3/4$, we can sum $N_0$ and then sum $N_1 \sim N_2$ by Cauchy-Schwarz to get \eqref{Sest}. If $a=3/4$, we use instead $C(\mathbf N,\mathbf L) \le c N_0 \Lmin^{1/2}$ and take $\mu=1/4$, yielding
\begin{equation}\label{Bound2}
  \sum_{\mathbf N, \mathbf L} \mathbb{1}_{N_0 \ll N_1 \sim N_2}
  \left( \frac{\Lmax}{N_1} \right)^{1/4}
  \frac{N_0 \Lmin^{1/2}}{N_0^{3/4} L_0^{b_0} L_1^{b_1} L_2^{b_2}}
  \alpha_{N_0,L_0} \beta_{N_1,L_1} \gamma_{N_2,L_2}.
\end{equation}
Now we use the fact that
\[
  \frac{\Lmin^{1/2} \Lmax^{1/4}}{L_0^{b_0} L_1^{b_1} L_2^{b_2}} \le 1
\]
if $b_0+b_1+b_2 \ge 3/4$ and $b_0,b_1,b_2 \ge 1/4$, which are consistent with the assumptions of the lemma when $a=3/4$, so we reduce to
\[
  \sum_{\mathbf N, \mathbf L} \mathbb{1}_{N_0 \ll N_1 \sim N_2}
  \left(\frac{N_0}{N_1}\right)^{1/4}
  \alpha_{N_0,L_0} \beta_{N_1,L_1} \gamma_{N_2,L_2},
\]
and again obtain the desired bound \eqref{Sest}.

It remains to consider the subcase $\Lmax = L_0$ of case (i). The argument used for $a=3/4$ above still applies and yields \eqref{Bound2}, so it remains to consider $a < 3/4$. If $s_1 \neq s_2$, then by Lemma \ref{BilinearLemma} (with signs $-s_1$ and $s_2$, so equal signs) we have the estimate $C(\mathbf N,\mathbf L) \le c N_0^{1/2} \Lmin^{1/2} \Lmed^{1/2}$. Interpolating this with $C(\mathbf N,\mathbf L) \le cN_0 \Lmin^{1/2}$ gives $C(\mathbf N,\mathbf L) \le c N_0^{3/4} \Lmin^{1/2} \Lmed^{1/4}$ and hence we get again \eqref{Bound1}.

This leaves us with $s_1 = s_2$ in case (i) with $\Lmax=L_0$. From \eqref{Angle} we have $\theta_{12} \le c N_0/N_1$, since $\abs{\xi} \lesssim N_0 \ll \abs{\eta} \sim N_1$. Interpolating this with \eqref{AngleEst} gives
\[
  \theta_{12} \le c \left(\frac{N_0}{N_1}\right)^{1-2\mu} \left( \frac{\Lmax}{N_1} \right)^\mu
\]
for $\mu \in [0,1/2]$. Invoking Lemma \ref{BilinearLemma} with $C(\mathbf N,\mathbf L) \le c N_0^{1/2} N_1^{1/4} \Lmin^{1/2} \Lmed^{1/4}$, we obtain the bound
\[
  \sum_{\mathbf N, \mathbf L} \mathbb{1}_{N_0 \ll N_1 \sim N_2}
  \frac{N_0^{3/2-2\mu-a}}{N_1^{3/4-\mu}}
  \frac{\Lmin^{1/2} \Lmed^{1/4} \Lmax^\mu}{L_0^{b_0} L_1^{b_1} L_2^{b_2}}
  \alpha_{N_0,L_0} \beta_{N_1,L_1} \gamma_{N_2,L_2}.
\]
Taking $\mu=3/4-a$ and applying \eqref{Lbound} we reduce to
\[
  \sum_{\mathbf N, \mathbf L} \mathbb{1}_{N_0 \ll N_1 \sim N_2}
  \left( \frac{N_0}{N_1} \right)^a
  \alpha_{N_0,L_0} \beta_{N_1,L_1} \gamma_{N_2,L_2},
\]
so we only need $a > 0$ to sum $N_0$, and then we sum $N_1 \sim N_2$ by Cauchy-Schwarz. This concludes the proof of case (i) and of Theorem \ref{TrilinearTheorem}.

\section{Local existence}\label{Local}

In this section we prove the following local existence result, which is an extended version of Theorem \ref{LWPa}.

\begin{theorem}\label{LWP}
There exist $c, c_0 > 0$ such that for any $\sigma_0 \ge 0$ and any data \eqref{Splitdata}, the Cauchy problem \eqref{DKG} has a unique local solution $(\psi_+,\psi_-,\phi_+) \in C([-\delta,\delta];X_0)$, where
\[
  \delta = \frac{c_0}{1+a_0^2+b_0^2},
  \quad
  a_0^2 = \norm{f_+}_{G^{\sigma_0,0}}^2 + \norm{f_-}_{G^{\sigma_0,0}}^2,
  \quad
  b_0 = \norm{g_+}_{G^{\sigma_0,1/2}}.
\]
Moreover,
\[
  \sum_\pm \norm{\psi_\pm}_{X_\pm^{\sigma_0,0,1/2;1}(\delta)} \le 2ca_0
  \quad \text{and} \quad
  \norm{\phi_+}_{X_+^{\sigma_0,1/2,1/2;1}(\delta)} \le 2c(a_0+b_0).
\]
\end{theorem}

\begin{proof}
To simplify the notation we write $\sigma=\sigma_0$. Define the Picard iterates $(\psi^{(n)}_+, \psi^{(n)}_-, \phi^{(n)}_+)_{n=-1}^\infty$ by starting at zero at $n=-1$ and continuing by the scheme
\[
\left\{
\begin{alignedat}{2}
  (-i\partial_t + \abs{D})\psi_+^{(n+1)} &= \Pi_+\bigl( - M\beta\psi^{(n)} + (\re\phi_+^{(n)})\beta\psi^{(n)}\bigr),& \quad \psi_+^{(n+1)}(0,x) &= f_+(x),
  \\
  (-i\partial_t - \abs{D})\psi_-^{(n+1)} &= \Pi_-\bigl(- M\beta\psi^{(n)} + (\re\phi_+^{(n)})\beta\psi^{(n)}\bigr),& \quad \psi_-^{(n+1)}(0,x) &= f_-(x),
  \\
  (-i\partial_t + \angles{D}) \phi_+^{(n+1)} &= \angles{D}^{-1} \innerprod{\beta \psi^{(n)}}{\psi^{(n)}},& \quad \phi_+^{(n+1)}(0,x) &= g_+(x),
\end{alignedat}
\right.
\]
where $\psi^{(n)} := \psi^{(n)}_+ + \psi^{(n)}_-$. Since $\Pi_+^2=\Pi_+$ and $\Pi_-\Pi_+ = 0$, we have $\Pi_+\psi^{(n)}_+ = \psi^{(n)}_+$ and $\Pi_-\psi^{(n)}_+ = 0$. Thus, $\Pi_+\psi^{(n)} = \psi^{(n)}_+$, and similarly $\Pi_-\psi^{(n)} = \psi^{(n)}_-$. Setting
\[
\begin{alignedat}{2}
  A_n &= \sum_\pm \bignorm{\psi^{(n)}_\pm}_{X^{\sigma,0,1/2;1}_\pm(\delta)},&
  \qquad
  \Delta A_n &= \sum_\pm \bignorm{\psi^{(n+1)}_\pm - \psi^{(n)}_\pm}_{X^{\sigma,0,1/2;1}_\pm(\delta)},
  \\
  B_n &= \bignorm{\phi^{(n)}_+}_{X^{\sigma,1/2,1/2;1}_+(\delta)},&
  \Delta B_n &= \bignorm{\phi^{(n+1)}_+ - \phi^{(n)}_+}_{X^{\sigma,1/2,1/2;1}_+(\delta)},
\end{alignedat}
\]
we claim that
\begin{align}
  \label{Aest}
  A_{n+1} &\le ca_0 + c \delta^{1/2} A_n \left( M + B_n \right),
  \\
  \label{Best}
  B_{n+1} &\le cb_0 + c \delta^{1/2} A_n^2,
\end{align}
and
\begin{align}
  \label{dAest}
  \Delta A_{n+1} &\le c \delta^{1/2} \Delta A_n \left( M + B_n \right) + c \delta^{1/2} A_n \Delta B_n,
  \\
  \label{dBest}
  \Delta B_{n+1} &\le c \delta^{1/2} A_n \Delta A_n.
\end{align}
Then by induction one obtains $A_n \le 2ca_0$ and $B_n \le 2c(a_0+b_0)$ for all $n$, and further $\Delta A_{n+1} + \Delta B_{n+1} \le (1/2) (\Delta A_n + \Delta B_n)$, with $\delta$ as in the statement of the theorem, for a sufficiently small $c_0 > 0$ depending on $c$ and $M$. The sequence of iterates therefore converges and the conclusion of the theorem follows.

It remains to prove the claimed estimates. By Lemma \ref{IVPlemma},
\begin{multline*}
  A_{n+1} \le ca_0 + \sum_\pm c \delta^{1/2} \bignorm{\Pi_{\pm}\bigl( M \beta \psi^{(n)} \bigr)}_{X^{\sigma,0,0;\infty}_\pm(\delta)}
  \\
  +
  \sum_\pm c \delta^{1/2-b_2} \bignorm{\Pi_{\pm}\bigl(\re\phi_+^{(n)}\beta\psi^{(n)}\bigr)}_{X^{\sigma,0,-b_2;\infty}_\pm(\delta)}.
\end{multline*}
Using $X^{\sigma,0,0;\infty}_\pm(\delta) = X^{\sigma,0,0;\infty}_\mp(\delta)$, the identity \eqref{Commutation} and Lemma \ref{CutoffLemma}, we bound the second term on the right by
\[
  \sum_\pm c \delta^{1/2} \bignorm{M \beta \psi^{(n)}_\mp}_{X^{\sigma,0,0;\infty}_\mp(\delta)}
  \le
  c_\varepsilon \delta^{1-\varepsilon} M A_n
\]
for any $\varepsilon > 0$. The third term we bound by, applying Corollary \ref{BilinearCor} with $a=1/2$,
\begin{equation}\label{DifficultEstimate}
  \sum_{s_1,s_2} c\delta^{1/2-b_2} \bignorm{\Pi_{s_2}\bigl(\re\phi_+^{(n)}\beta\Pi_{s_1}\psi^{(n)}\bigr)}_{X^{\sigma,0,-b_2;\infty}_\pm(\delta)}
  \le c
  \delta^{3/2-b_0-b_1-b_2} B_n A_n,
\end{equation}
which requires $b_0,b_1,b_2 \ge 1/4$ and $b_0+b_1+b_2 \ge 1$. We choose $b_0=b_1=b_2=1/3$. Finally, the estimate \eqref{Best} similarly reduces to
\[
  \bignorm{\innerprod{\beta \Pi_{s_1}\psi^{(n)}}{\Pi_{s_2}\psi^{(n)}}}_{X^{\sigma,-1/2,-1/3;\infty}_+(\delta)}
  \le c
  \delta^{1/3} A_n^2,
\]
which also follows from Corollary \ref{BilinearCor}. Finally, the estimates \eqref{dAest} and \eqref{dBest} follow from the same considerations by linearity.
\end{proof}

\section{Approximate conservation of charge}\label{Approx}

In this section we prove Theorem \ref{ApproxCons}. We need the following key estimate.

\begin{lemma}\label{CommutatorLemma} Assume that $a,b_0,b_1,b_2$ satisfy the assumptions of Theorem \ref{TrilinearTheorem}. Then there exists a constant $c > 0$ such that for all signs $s_0,s_1,s_2 \in \{+,-\}$, all $\sigma \ge 0$ and all $\theta \in [0,1]$ we have the estimate
\begin{multline*}
  \Abs{ \int \Innerprod{e^{\sigma\norm{D}}\bigl(\phi\beta\Pi_{s_1}\psi_1\bigr) - \phi\bigl(\beta\Pi_{s_1}e^{\sigma\norm{D}}\psi_1\bigr)}{\Pi_{s_2}\psi_2} \, dt \, dx}
  \\
  \le c \sigma^{\theta}
  \norm{\phi}_{X^{\sigma,a+\theta,b_0;1}_{s_0}}
  \norm{\psi_1}_{X^{\sigma,0,b_1;1}_{s_1}}
  \norm{\psi_2}_{X^{0,0,b_2;1}_{s_2}}.
\end{multline*}
\end{lemma}

\begin{proof}
By Plancherel's theorem we bound the left side by
\[
  \Abs{ \int \Lambda(\xi,\eta) \widehat\phi(\tau,\xi)
  \Innerprod{\beta\Pi\left(s_1(\eta-\xi)\right)\widehat{\psi_1}(\lambda-\tau,\eta-\xi)}
  {\Pi\left(s_2\eta\right)\widehat{\psi_2}(\lambda,\eta)} \, d\lambda \, d\tau \, d\eta \, d\xi }
\]
where
\[
  \Lambda(\xi,\eta)
  =
  e^{\sigma\norm{\eta}} - e^{\sigma\norm{\eta-\xi}}
  =
  e^{\sigma\norm{\eta-\xi}} \left( e^{\sigma(\norm{\eta} - \norm{\eta-\xi})} - 1 \right).
\]
As in the proof of Theorem \ref{TrilinearTheorem} we then bound by
\[
  c \int \Abs{\Lambda(\xi,\eta)} \theta_{12}
  \Abs{\widehat\phi(\tau,\xi)}
  \Abs{\widehat{\psi_1}(\lambda-\tau,\eta-\xi)}
  \Abs{\widehat{\psi_2}(\lambda,\eta)} \, d\lambda \, d\tau \, d\eta \, d\xi
\]
Applying the inequality
\[
  \Abs{e^x-1} \le \abs{x}^\theta e^{\abs{x}} \quad (x \in \R, \; \theta \in [0,1]),
\]
and the triangle inequality $\bigabs{\norm{\eta} - \norm{\eta-\xi}} \le \norm{\xi}$, we finally bound by
\[
  c \sigma^\theta \int \theta_{12}
  \angles{\xi}^\theta e^{\sigma\norm{\xi}} \Abs{\widehat\phi(\tau,\xi)}
  e^{\sigma\norm{\eta-\xi}} \Abs{\widehat{\psi_1}(\lambda-\tau,\eta-\xi)}
  \Abs{\widehat{\psi_2}(\lambda,\eta)} \, d\lambda \, d\tau \, d\eta \, d\xi
\]
and the desired estimate then follows from \eqref{FourierEst}.
\end{proof}

We now have all the tools needed to prove the approximate conservation law.

\subsection{Proof of Theorem \ref{ApproxCons}}

By Theorem \ref{LWP} (applied with $\sigma_0$ replaced by $\sigma \in [0,\sigma_0]$) there exist constants $c,c_0 > 0$ such that for all $\sigma \in [0,\sigma_0]$ we have the bounds
\begin{gather}
  \label{IterationBound1}
  \norm{\psi_+}_{X^{\sigma,0,1/2;1}_+(\delta(\sigma))} + \norm{\psi_-}_{X^{\sigma,0,1/2;1}_-(\delta(\sigma))}
  \le c\mathfrak M_\sigma(0)^{1/2},
  \\
  \label{IterationBound2}
  \norm{\phi_+}_{X^{\sigma,1/2,1/2;1}_+(\delta(\sigma))}
  \le c \left( \mathfrak M_\sigma(0)^{1/2} + \mathfrak N_\sigma(0) \right),
\end{gather}
where
\begin{equation}\label{Time}
  \delta(\sigma) = \frac{c_0}{1+\mathfrak M_\sigma(0) + \mathfrak N_{\sigma}(0)^2}.
\end{equation}
But clearly, $\delta(\sigma) \ge \delta := \delta(\sigma_0)$ for $\sigma \in [0,\sigma_0]$, so we may replace $\delta(\sigma)$ by $\delta$ in \eqref{IterationBound1} and \eqref{IterationBound2}. 

\subsubsection{Proof of \eqref{Mest}}

Set $\Psi_\pm = e^{\sigma\norm{D}}\psi_\pm$ and $\Psi=\Psi_+ + \Psi_-$. Then \eqref{DKG} gives
\begin{align*}
  (-i\partial_t + \abs{D})\Psi_+ = \Pi_+\bigl( - M\beta\Psi + (\re\phi_+)\beta\Psi\bigr) + \Pi_+F,
  \\
  (-i\partial_t - \abs{D})\Psi_- = \Pi_-\bigl(- M\beta\Psi + (\re\phi_+)\beta\Psi\bigr) + \Pi_-F,
\end{align*}
where
\[
  F = e^{\sigma\norm{D}} \bigl( (\re\phi_+)\beta\psi \bigr) - (\re\phi_+)\beta\Psi.
\]
Now we calculate
\begin{align*}
  \frac{d}{dt} \mathfrak M_\sigma(t)
  &=
  \frac{d}{dt} 
  \int \bigl( \innerprod{\Psi_+(t,x)}{\Psi_+(t,x)} + \innerprod{\Psi_-(t,x)}{\Psi_-(t,x)} \bigr) \, dx
  \\
  &=
  2 \im
  \int \bigl( \innerprod{i\partial_t\Psi_+}{\Psi_+} + \innerprod{i\partial_t\Psi_-}{\Psi_-} \bigr) \, dx
  \\
  &=
  2 \im
  \int \bigl( \innerprod{(i\partial_t-\abs{D})\Psi_+}{\Psi_+} + \innerprod{(i\partial_t+\abs{D})\Psi_-}{\Psi_-} \bigr) \, dx
  \\ & \quad + 2 \im \int \bigl( \innerprod{\abs{D}\Psi_+}{\Psi_+} + \innerprod{-\abs{D}\Psi_-}{\Psi_-} \bigr) \, dx
  \\
  &=
  2 \im
  \int \bigl( (M-\re\phi_+)\innerprod{\beta\Psi}{\Psi} - \innerprod{F}{\Psi} \bigr) \, dx
  \\
  &= - 2 \im
  \int \innerprod{F}{\Psi} \, dx,
\end{align*}
where we used Plancherel to see that $\int \innerprod{\abs{D}\Psi_\pm}{\Psi_\pm} \, dx = 0$ and the self-adjointness of $\beta$ to see that $\innerprod{\beta\Psi}{\Psi}$ is real valued. Integrating over the time interval $[0,T]$ for any $T \in [0,\delta]$ we then get
\[
  \mathfrak M_\sigma(T)
  \le
  \mathfrak M_\sigma(0)
  +
  2 \Abs{ \int \chi_{[0,T]}(t) \innerprod{F}{\Psi}(t,x) \, dx \, dt}
\]
and applying Lemma \ref{CommutatorLemma} with
\begin{equation}\label{bchoice}
  b_0=b_1=b_2=
  \begin{cases} 1/2-a/3 &\text{if $a \in [3/8,1/2]$}
  \\
  3/4-a &\text{if $a \in (1/4,3/8)$}
  \end{cases},
\end{equation}
 we bound the integral term by
\[
  c \sum_{s_0,s_1,s_2 \in \{+,-\}} \sigma^{\theta}
  \norm{\chi_{[0,T]}\phi_{s_0}}_{X^{\sigma,a+\theta,b_0;1}_{s_0}}
  \norm{\chi_{[0,T]}\psi_{s_1}}_{X^{\sigma,0,b_1;1}_{s_1}}
  \norm{\chi_{[0,T]}\psi_{s_2}}_{X^{\sigma,0,b_2;1}_{s_2}},
\]
where we wrote $2\re\phi_+ = \phi_+ + \overline{\phi_+}$ and used $\phi_- = \overline{\phi_+}$. Taking $\theta=1/2-a$ and invoking Lemma \ref{SharpTimeRestrictionLemma} followed by Lemma \ref{CutoffLemma}, we bound the summands by
\begin{align*}
  &c \sigma^{1/2-a}
  \norm{\phi_{s_0}}_{X^{\sigma,1/2,b_0;1}_{s_0}(T)}
  \norm{\psi_{s_1}}_{X^{\sigma,0,b_1;1}_{s_1}(T)}
  \norm{\psi_{s_2}}_{X^{\sigma,0,b_2;1}_{s_2}(T)}
  \\
  &\quad \le c T^{3/2-b_0-b_1-b_2}
  \sigma^{1/2-a}
  \norm{\phi_{s_0}}_{X^{\sigma,1/2,1/2;1}_{s_0}(T)}
  \norm{\psi_{s_1}}_{X^{\sigma,0,1/2;1}_{s_1}(T)}
  \norm{\psi_{s_2}}_{X^{\sigma,0,1/2;1}_{s_2}(T)}
  \\
  &\quad \le c T^{p}
  \sigma^{1/2-a}
  \mathfrak M_\sigma(0) \left( \mathfrak M_\sigma(0)^{1/2} + \mathfrak N_\sigma(0) \right),
\end{align*}
where we applied the bounds \eqref{IterationBound1} and \eqref{IterationBound2} and used the fact that $\norm{\phi_{-}}_{X^{\sigma,1/2,1/2;1}_{-}(T)} = \norm{\phi_+}_{X^{\sigma,1/2,1/2;1}_{+}(T)}$, on account of $\phi_- = \overline{\phi_+}$. This concludes the proof of \eqref{Mest}.

\subsubsection{Proof of \eqref{Nest}}

Applying Corollary \ref{IVPcor} to the last equation in \eqref{DKG} gives
\[
  \sup_{t \in [0,\delta]} \mathfrak N_\sigma(t)
  \le
  \mathfrak N_\sigma(0)
  + c\delta^{1/2-b_0} \norm{\angles{D}^{-1} \innerprod{\beta \psi}{\psi}}_{X^{\sigma,1/2,-b_0;\infty}_{+}(\delta)},
\]
where $b_0 \in [0,1/2]$ remains to be chosen. Separating low frequencies, $\norm{\xi} \le \sigma^{-1}$, and high frequencies, $\norm{\xi} > \sigma^{-1}$, we estimate the last term by
\[
  c\delta^{1/2-b_0} \left( \norm{\angles{D}^{-1} \innerprod{\beta \psi}{\psi}}_{X^{0,1/2,-b_0;\infty}_{+}(\delta)}
  +
  \sigma^{\theta} \norm{\angles{D}^{\theta-1} \innerprod{\beta \psi}{\psi}}_{X^{\sigma,1/2,-b_0;\infty}_{+}(\delta)} \right),
\]
where $\theta \in [0,1]$ remains to be chosen. We are going to estimate both terms using Corollary \ref{BilinearCor} and the bound \eqref{IterationBound1}. First, taking $a=1/2$ and $b_1=b_2=(1-b_0)/2$ for any $b_0 \in [0,1/2]$, and setting $\sigma=0$, we get
\[
  \delta^{1/2-b_0}\norm{\angles{D}^{-1} \innerprod{\beta \psi}{\psi}}_{X^{0,1/2,-b_0;\infty}_{+}(\delta)}
  \le c \delta^{3/2-b_0-b_1-b_2} \mathfrak M_0(0)
  = c \delta^{1/2} \norm{\psi(0,\cdot)}_{L^2}^2.
\]
Taking $a=1/2-\theta$ and choosing the $b$'s as in \eqref{bchoice}, we similarly bound
\[
  \delta^{1/2-b_0} \sigma^{\theta} \norm{\angles{D}^{\theta-1} \innerprod{\beta \psi}{\psi}}_{X^{\sigma,1/2,-b_0;\infty}_{+}(\delta)}
  \le
  c \delta^{p} \sigma^{1/2-a} \mathfrak M_\sigma(0),
\]
concluding the proof of \eqref{Nest} and of Theorem \ref{ApproxCons}.

\bibliographystyle{amsplain}
\bibliography{database}

\end{document}